\documentclass[11pt]{article}
\usepackage[utf8]{inputenc}

\usepackage[english]{babel}
\usepackage{amsfonts}
\usepackage{amsmath,amssymb, amsthm}
\usepackage{enumerate}
\usepackage{tikz}
\usepackage{graphicx}
\usepackage{caption}
\usepackage{subcaption}

\usepackage{t1enc}
\usepackage{ifthen}
\usepackage{amssymb}
\usepackage[normalem]{ulem}
\usepackage{subfloat}

\usepackage{tabularx}

\usepackage{algorithm}
\usepackage{algpseudocode}
\algnewcommand{\IIf}[1]{\State\algorithmicif\ #1\ \algorithmicthen}
\algnewcommand{\EndIIf}{\unskip\ \algorithmicend\ \algorithmicif}
\title{Exhaustive search of convex pentagons which tile the plane}
\author {Michaël Rao}

\newtheorem{theorem}{Theorem}
\newtheorem{lemma}[theorem]{Lemma}

\newtheorem{proposition}[theorem]{Proposition}
\theoremstyle{definition}
\newtheorem*{example}{Example}
\begin{document}
\maketitle

\begin{abstract}
  We present an exhaustive search of all families of convex pentagons which tile the plane.
  This research shows that there are no more than the already 15 known families. 
  In particular, this implies that there is no convex polygon which allows only non-periodic tilings.
\end{abstract}

\def\note#1{}%

\def\plane{\mathbb{R}^2}
\def\surf{\mathcal{S}}
\def\seq{\mathfrak{s}}
\def\penta{\mathcal{P}}
\def\ver{s} %
\def\tiling{\mathcal{T}}
\def\nodes#1{\mathcal{T}(#1)} %
\def\fnodes#1{\mathcal{FT}(#1)} %
\def\vertex#1{\mathcal{V}(#1)} %
\def\ivertex#1{\mathcal{IV}(#1)} %
\def\ivertexh#1{\mathcal{IV}_H(#1)} %
\def\ivertexf#1{\mathcal{IV}_F(#1)} %
\def\vect{V} %
\def\cvect{{V^c}} %

\def\mfA{\mathfrak{A}}
\def\mX{\mathcal{X}}
\def\mY{\mathcal{Y}}
\def\mB{\mathcal{B}}

\def\v1{(1,1,1,1,1)}

\def\setvectype{\mathcal{W}} %
\def\setvecctype{{{\mathcal{W}}^c}} %

\section{Introduction}\label{sec:positive}

If one asks which convex polygon can tile the plane (allowing translations, rotations and mirrors), the case of pentagons is the only opened case: every triangle and quadrilateral tiles the plane, there are 3 families of hexagons which tile the plane, and no convex polygon with more than six sides can tile the plane (see for example \cite{Schattschneider}).

The research of families of pentagons which tile the plane has an intriguing history.
The first families of pentagons were presented by Reinhardt in 1918  \cite{Reinhardt}.
Kershner presented new families, and announced that the list was complete in 1968 \cite{Kershner}. But new families were found afterwards, one by R. James in 1975, three by an amateur mathematician M. Rice in 1977, and one by R. Stein in 1985. Finally, the fifteenth (and last) family was found by Mann, McLoud and Von Derau in 2015 (see \cite{last}).

We present here an exhaustive search of all families of pentagons which tile the plane. This search is not restricted to periodic tilings, and does not find any new family.
The key point is that there are only finitely many, 371, families of angle conditions to consider.

In Section \ref{sec:positive}, we introduce the notations, and we show that if a pentagon tiles the plane, then there is a tiling such that every vertex type has positive density.
In Section \ref{sec:371}, we show that there are only finitely possible sets of vertex types in a positive density tiling by a pentagon.
Then, in Section \ref{sec:phase3}, we present a backtracking technique to search a tiling, when we fix the set of vertex types.
This backtracking algorithm does not find any new family of pentagons. 

\section{Positive density tilings}\label{sec:positive}

Throughout this section, we fix a convex pentagon $\penta\subset\plane$.
Let $s_1,\ldots s_5 \in \plane$ be its 5 vertices in clockwise order.
For $i\in \{1,..5\}$, let $\alpha_i\times\pi$ be the angle at vertex $s_i$
and let $\alpha=(\alpha_1,\ldots, \alpha_5)$.
We recall that $\sum_{i=1}^5 \alpha_i = 3$.

A \emph{tiling} $\tiling$ of $\surf$ by $\penta$ is a set of subset of $\surf$ such that:
\begin{itemize}
\item $\cup_{P\in \tiling} P = \surf$,
\item for every $P\in \tiling$, there is an isometry of the plane $h_P$ such that $h_P(\penta)=P$,
\item for every $P,Q\in \tiling$ with $P\ne Q$, $\mathring{P} \cap \mathring{Q} = \emptyset$ (where $\mathring{P}$ is the interior of $P$ with the usual topology on $\plane$).
\end{itemize}
Given a tiling of $\surf$, we fix an isometry $h_P$ for every $P\in \tiling$ such that $P=h_P(\penta)$. %
Elements of $\tiling$ are called a \emph{tiles}.
A point $\ver$ in $\surf$ is a vertex of $\tiling$ if it is a vertex of at least one tile in $\tiling$ (that is, there is a $P\in \tiling$ and a $i\in\{1,\ldots,5\}$ such that $\ver = h_P(s_i)$.
The set of vertices of $\tiling$ is denoted $\vertex{\tiling}$.
A \emph{tiling of the plane} is a tiling of $\plane$.

From now on, we fix a tiling $\tiling$ of the plane by $\penta$.
Let $P,Q\in \tiling$ and $\ver \in \vertex{\tiling}$. We say that \emph{$Q$ follows $P$ around $\ver$} if there are $i$ and $j$ such that $h_P(s_i)=h_Q(s_j)=\ver$, $P\cap Q \ne \{\ver\}$,
and $Q$ are just after $P$ if we turn around $\ver$ in clockwise order.

We distinguish two types of vertices: full and half.
Let $\ver$ be a vertex of a tiling $\tiling$. %
If there is a circular sequence $\seq_\ver=(P_1,\ldots, P_k)$ of tiles such that for every $i\in\{1,\ldots k\}$, $P_{i+1}$ follows $P_i$ around $\ver$ (where the indices are taken modulo $k$) we say that $\ver$ is \emph{full}. 

Otherwise, we say that $\ver$ is \emph{half}. In this case, there is a tile $Q$ for which $\ver$ is on the border, but $\ver$ is not a vertex of $Q$. That is, $\ver \in Q \setminus (\mathring{Q} \cup \{s_1,\ldots, s_5\})$, and $\ver$ is on the line segment from $h_Q(s_i)$ to $h_Q(s_{i+1})$ for a $i\in\{1,\ldots 5\}$ (modulo 5). Then there is a maximal sequence $\seq_\ver=(P_1,\ldots, P_k)$ of tiles such that for every $i\in \{1,\ldots k-1\}$, $P_{i+1}$ follows $P_i$ around $\ver$. (There is no tile $Q$ such that $Q$ follows $P_k$ around $\ver$, or $P_1$ follows $Q$ around $\ver$.) Note that for any vertex $\ver$, if $\ver$ is full, then the circular sequence $\seq_\ver$ is unique, and if $\ver$ is half, the maximal sequence $\seq_\ver$ is also unique.

The \emph{vector type} of $\ver$, denoted $\vect(\ver)$, is $(c_1,\ldots,c_5) \in \mathbb{N}^5$, where for every $i\in\{1,\ldots 5\}$, $c_i$ is the number of tiles $P$ in $\seq_\ver$ such that $h_P(s_i)=\ver$.
The \emph{corrected vector type} of $\ver$, denoted $\cvect(\ver)$, is either $\vect(\ver)$ if $\ver$ is full, or $2\times \vect(\ver)$ if $\ver$ is half.
We have in any case $\cvect(\ver) \cdot \alpha =2$.

\def\arcs{\mathcal{A}}
Let $G=(\tiling,\arcs)$ be the oriented graph, called the \emph{underliyng graph} of $\tiling$, such that there is an arc between $P$ and $Q$ if $P$ follows $Q$ around $\ver$ for a vertex $\ver$ in $\tiling$. 
Moreover, we label each tile $P$ in the graph by ``$+$'' if $h_P$ is a translation or a rotation, or by ``-'' if $h_P$ is a glide reflexion.
We label each arc $(P,Q)$ by $(i,j)$ with $h_P(s_i) = h_Q(s_j) =\ver$, where $\ver$ is such that $P$ follows $Q$ around $\ver$.

Let $\tiling'\subset\tiling$. The subgraph induced by $\tiling'$ is the graph $G[\tiling']=(\tiling',\arcs')$, where $\arcs'=\arcs \cap \tiling^2$.
A tile $P$ in $\tiling'$ is a \emph{frontier tile} if there is a $Q\in\tiling \setminus \tiling'$ such that either $(P,Q)\in \arcs$ or $(Q,P)\in \arcs$. %
The set of frontier tiles of a subgraph $H$ is denoted $\fnodes{H}$.

Given an induced subgraph $H$ of $G$, we denote by $\surf(H) = \bigcup_{P\in \tiling(H)} P$, where $\tiling(H)$ is the set of tiles in $H$. (Note that $\tiling(H)$ is then a tiling of $\surf(H)$.)

The set of vertices of $H$ is denoted $\vertex{H}=\vertex{\tiling(H)}$.
A vertex $\ver$ is an \emph{interior vertex} of a subgraph $H$ of $G$ if for every $P\in \tiling$ such that $\ver$ is a vertex of $P$, then $P\in \nodes{H}$.
The set of interior vertices of $H$ is denoted $\ivertex{H}$.
Moreover, the set of interior and half (resp. full) vertices of $H$ is denoted $\ivertexh{H}$ (resp. $\ivertexf{H}$).
Note that an interior half vertex of $H$ can be on the boundary of $\surf(H)$.

Let $\setvectype(\tiling)=\{\vect(\ver) : \ver \in \vertex{\tiling\}}$ and $\setvecctype(\tiling)=\{\cvect(\ver) : \ver \in \vertex{\tiling}\}$. 
If $\tiling$ is clear in the context, we may write $\setvectype$ or $\setvecctype$.
Note that $\setvectype(\tiling)$ and $\setvecctype(\tiling)$ are finite.

Let $o\in \plane$ and $r\in \mathbb{R}^+$.
Let $\tiling_{o,r}$ be the set of tiles $P\in \tiling$ such that $P\cap D(o,r) \ne \emptyset$, where $D(o,r)$ is the closed disk of radius $r$ and center $o$.
Let $G_{o,r}= G[\tiling_{o,r}]$ be the graph induced by $\tiling_{o,r}$.

\begin{proposition}\label{prop:o}
There are constants $c$ and $c'$ in $\mathbb{R}^+$ such that for every $r\in\mathbb{R}^+$, $|\fnodes{G_{o,r}}| \le c \times r + c'$.  
\end{proposition}
\begin{proof}
This follows from the fact that for every $r\in \mathbb{R}^+$, there is a $n_r\in \mathbb{N}$ such that for every $o\in \plane$, $\vert \tiling_{o,r} \vert \le n_r$.
\end{proof}

For $v\in \setvectype(\tiling)$, $o\in \plane$ and $r\in\mathbb{R}^+$, let: $$f_{o,r}(v)=\frac{|\{ \ver\in \ivertex{G_{o,r}} : \vect(\ver) = v \}|}{|\tiling_{o,r}|}.$$

We say that the tiling $\tiling$ has \emph{positive density} if for every $v\in \setvectype(\tiling)$ and $o\in\plane$, we have $\liminf_{r\to\infty} f_{o,r}(v) > 0$.
(Note that if it is true for one $o\in\plane$, then it is true for every $o\in\plane$.)

\begin{lemma}\label{lemma:positivedensity}
If $\liminf_{r\to\infty} f_{o,r}(v) = 0$, then there is a tiling $\tiling'$ of the plane by $\penta$ such that $\setvectype(\tiling')\subseteq\setvectype(\tiling)\setminus\{v\}$. %
\end{lemma}
\begin{proof}
  Let $d\in \mathbb{R}^+$. We divide $\plane$ into a grid of $d\times d$ squares $S_{(i,j)}$, with $(i,j)\in\mathbb{Z}^2$. Then we decompose $\tiling$ into a disjoint union of sets of tiles $\tiling_{(i,j)}$ such that a tile $P$ is in $\tiling_{(i,j)}$ if $P\cap S_{(i,j)} \ne \emptyset$ (if there are several possible choices for a tile, we chose arbitrarily). If for every $(i,j)\in \mathbb{Z}^2$, $\tiling_{(i,j)}$ has one vertex $\ver$ with $\vect(\ver)=v$, then $\liminf_{r\to\infty} f_{o,r}(v) > 0$. %

Thus, if $\liminf_{r\to\infty} f_{o,r}(v) = 0$, then for every $d\in\mathbb{R}^+$ there is a subgraph $H_d$ such that $\surf(H_d)$ contains a $d\times d$ square, and $\setvectype(H_d)\subseteq\setvectype(\tiling)\setminus\{v\}$.
We keep a connected component $H'_d$ of $H_d$ such that $\surf(H'_d)$ contains the center of the square.
Then one can construct by compactness an infinite graph $G'$ in which every vertex is an interior vertex, and with $\setvectype(G')\subseteq\setvectype(\tiling)\setminus\{v\}$.
There are three cases: either $G'$ corresponds to a tiling of the plane, of a half-plane or of a stripe. In all cases, one can construct a tiling of the plane without vertex of vector type $v$, and no new vector type.
\end{proof}

\paragraph{Good subsets.} %
We say that a subset $\mX\subseteq \mathbb{N}^5$ is \emph{good} if for every $u\in \mathbb{R}^5$ with $\sum u=0$, either $u\cdot v=0$ for every $v\in \mX$, or there are $v,v'\in \mX$ such that $u\cdot v<0<u\cdot v'$.

\def\poly{\mathfrak{P}}

Suppose that $\tiling$ is a tiling by the pentagon $\penta$. %
By Lemma~\ref{lemma:positivedensity}, we assume that $\tiling$ has positive density.

\begin{proposition}\label{prop:good}
  Let $u\in\mathbb{R}^5$ such that $\sum_{i=1}^5 u_i=0$, and there is a $v^+\in\setvecctype$ with $u\cdot v^+>0$. Then there is a $v^-\in\setvecctype$ such that $u\cdot v^-<0$.
\end{proposition}
\begin{proof}
  For every $i\in\{1,\ldots 5\}$, if we count the number of angles $i$ in $\tiling(G_{o,r})$, we have:
  
  $$0 \le |\nodes{G_{o,r}}| - \sum_{v\in\setvectype} v_i \times |\{ \ver\in \ivertex{G_{o,r}} : \vect(\ver) = v \}| \le |\fnodes{G_{o,r}}|. $$
  
  Thus, by Proposition~\ref{prop:o}: $$ \lim_{r\to\infty} \sum_{v\in\setvectype} v \times f_{o,r}(v) = \v1.$$ %
  
  Since $\tiling$ has positive density, $\liminf_{r\to\infty} f_{o,r}(v)>0$ for every $v\in\setvectype$.
  Let $$f'_{o,r}(v)=\frac{1}{2}\frac{|\{ \ver\in \ivertexh{G_{o,r}} : \vect(\ver) = \frac{1}{2} v \}|}{|\nodes{G_{o,r}}|} + \frac{|\{ \ver\in \ivertexf{G_{o,r}} : \vect(\ver) = v \}|}{|\nodes{G_{o,r}}|}.$$
  Since $v\in\mathbb{N}^5$ cannot be the vector type of both a half vertex, and a full vertex in the same tiling, we have also $\liminf_{r\to\infty} f'_{o,r}(v)>0$ for every $v\in\setvecctype$.
  Moreover $$\lim_{r\to\infty} \sum_{v\in\setvecctype} v \times f'_{o,r}(v) = \v1 .$$

  Let $u\in \mathbb{R}^5$ and $v^+\in \setvecctype$ such that $u\cdot \v1=0$ and $u\cdot v^+>0$.
  Suppose for the sake of contradiction that for every $v'\in \setvecctype$, $u\cdot v'\ge 0$.
  Then $\lim_{r\to\infty} \sum_{v\in\setvecctype} (u\cdot v) \times f'_{o,r}(v)=0$.
  We have a contradiction since $\lim_{r\to\infty} \sum_{v\in\setvecctype} (u\cdot v) \times f'_{o,r}(v) \ge (u\cdot v^+)\times \liminf_{r\to\infty} f'_{o,r}(v^+)>0$.
\end{proof}

By Proposition \ref{prop:good}, if $\tiling$ has positive density, then $\setvecctype$ is good. %

\def\span{\operatorname{span}}
\def\compat{\operatorname{Compat}} 
\paragraph{Compatible vectors.} %
Let $\span(V)$ be the set of vectors which are linear combinaisons of vectors in $V$.
Let $$\compat(V)=\{ w \in \mathbb{N}^5 : (w,2) \in \span(\{(1,1,1,1,1,3)\} \cup \{ (v,2) : v\in V \}) \}.$$
Note that if $\mX$ is good, then $\compat(\mX)$ is also good (but the converse is not necessarily true).

Given a subset $\mX\subseteq \mathbb{N}^5$, we define by $\poly_\mX$ the subset of $\mathbb{R}^5$ such that $\alpha=(\alpha_1,\ldots \alpha_5) \in \poly_\mX$ if and only if for every $i\in\{1,\ldots, 5\}$, $0 \le \alpha_i \le 1$, $\sum_{i=1}^5 \alpha_i = 3$ and for every $v\in \mX$, $\alpha\cdot v=2$.
One has $\poly_\mX = \poly_{\compat(\mX)}$.

If $\tiling$ is a tiling by a convex pentagon $\penta$ of angles $(\alpha_1\cdot \pi,\ldots \alpha_5\cdot \pi)$, then $({\alpha_1},\ldots {\alpha_5})\in \poly_{\setvecctype(T)} \cap ]0,1[^5$.
Moreover, if $\poly_\mX \cap ]0,1[^5 \ne \emptyset$, then the set $\compat(\mX)$ is finite.

In the next section, we compute all good sets $\mX$ with $\poly_\mX \cap ]0,1[^5 \ne \emptyset$. We show in particular that there are finitely many such sets.
    
\section{Computation of all good subsets}\label{sec:371}

\def\genset{\mB}

We say that the permutation in $S_5$ is a \emph{rotation/mirror} if it can be generated by the permutations $(12345)$ and $(3)(24)(15)$.
Given a permutation $p\in S_5$ and a vector $v\in \mathbb{N}^5$, let $p(v)$ be the vector $(v_{p(1)},\ldots, v_{p(5)})$. Let $p(V)$ for $V\subseteq \mathbb{N}^5$, be $\{ p(v) : v\in V \}$.
In this section, we show the following:
\begin{lemma}\label{lem:371}
  If $\mX$ is a non empty good set such that $\poly_{\mX}\cap ]0,1[^5\ne \emptyset$ then $\poly_{\mX} = \poly_{r(\compat(\genset_i))}$ for an integer $i\in \{1,\dots 371\}$ and a rotation/mirror $r$. ($\genset_i$ is given in Tables \ref{tab:371}.)
\end{lemma}

{
\begin{subtables}
  \label{tab:371}
  
  \begin{table}
    \scriptsize
    \begin{center}
      \begin{tabular}{|l|l||l|l|}
\hline
 i & $\genset_i$  &  i & $\genset_i$ \\\hline
\hline
\sout{1} &11100 &
{2} &11010 \\
\hline
\end{tabular}

    \end{center}
    \caption{$\dim(\poly)=3$. Striked out numbers correspond to families of pentagons of Type 1.}
  \end{table}

  \begin{table}
    \scriptsize
    \begin{center}
      \begin{tabular}{|l|l||l|l||l|l|}
\hline
 i & $\genset_i$  &  i & $\genset_i$  &  i & $\genset_i$ \\\hline
\hline
\sout{3} &00003 11100 &
{4} &00003 11010 &
\sout{5} &11100 00004 \\
{6} &11010 00004 &
\sout{7} &00012 00111 &
{8} &00012 11010 \\
{9} &00012 01011 &
\sout{10} &00012 01110 &
{11} &00012 01200 \\
{12} &00012 21000 &
{13} &00012 10110 &
{14} &00012 02100 \\
{15} &00012 20100 &
\sout{16} &00012 10011 &
{17} &00012 10200 \\
\sout{18} &00102 11100 &
\sout{19} &11010 11100 &
\sout{20} &11001 11100 \\
\sout{21} &01002 11100 &
\sout{22} &10110 11100 &
{23} &00102 01020 \\
{24} &00102 10110 &
{25} &00102 02010 &
{26} &00102 10101 \\
{27} &00102 10020 &
{28} &01011 11010 &
 & \\\hline
\end{tabular}

    \end{center}
    \caption{$\dim(\poly)=2$}
  \end{table}

  \begin{table}
    \tiny
    \begin{center}
      \begin{tabular}{|l|l||l|l||l|l|}
\hline
 i & $\genset_i$  &  i & $\genset_i$  &  i & $\genset_i$ \\\hline
\hline
\sout{29} &00003 00012 00102 &
\sout{30} &00003 00012 11001 &
{31} &00003 00012 01002 \\
{32} &00003 00012 01200 &
{33} &00003 00012 10101 &
{34} &00003 00012 02100 \\
{35} &00003 00012 20100 &
\sout{36} &00003 00102 11001 &
\sout{37} &00003 11010 11100 \\
\sout{38} &00003 01002 10101 &
\sout{39} &00003 10110 11100 &
\sout{40} &00003 01110 11100 \\
\sout{41} &00003 01200 11100 &
\sout{42} &00003 02100 11100 &
\sout{43} &00003 10200 11100 \\
{44} &00003 00102 01020 &
{45} &00003 00102 21000 &
{46} &00003 00102 02010 \\
{47} &00003 02010 20100 &
{48} &00003 02100 20010 &
{49} &00003 00120 21000 \\
{50} &00003 00120 10110 &
{51} &00003 00120 12000 &
{52} &00003 10110 11010 \\
{53} &00003 02010 11010 &
{54} &00003 10020 11010 &
{55} &00003 00210 12000 \\
{56} &00003 01200 20010 &
{57} &00003 01020 20100 &
{58} &00003 02010 10200 \\
\sout{59} &11001 11010 11100 &
\sout{60} &10101 10110 11100 &
\sout{61} &01200 11100 00004 \\
\sout{62} &02100 11100 00004 &
{63} &01020 11010 00004 &
{64} &02010 11010 00004 \\
\sout{65} &00111 01020 10200 &
\sout{66} &00111 01200 10020 &
{67} &00120 01011 12000 \\
\sout{68} &00120 10011 21000 &
{69} &00210 01101 12000 &
{70} &01011 02100 10020 \\
\sout{71} &01020 10011 20100 &
{72} &01101 02010 10200 &
{73} &00012 00102 01011 \\
\sout{74} &00012 00102 01020 &
{75} &00012 00102 21000 &
{76} &00012 00102 10020 \\
\sout{77} &00012 00102 10011 &
{78} &00012 00102 12000 &
\sout{79} &00012 00111 01002 \\
\sout{80} &00012 00111 01011 &
\sout{81} &00012 00111 02001 &
{82} &00012 00120 21000 \\
{83} &00012 00120 10002 &
{84} &00012 00120 12000 &
{85} &00012 01101 21000 \\
{86} &00012 02001 20100 &
{87} &00012 10101 12000 &
{88} &00012 01002 20100 \\
{89} &00012 02100 10002 &
{90} &00012 00201 01011 &
{91} &00012 00201 21000 \\
\sout{92} &00012 00201 10011 &
{93} &00012 00201 12000 &
\sout{94} &00012 01011 10011 \\
{95} &00012 10110 11010 &
\sout{96} &00012 01110 11010 &
{97} &00012 02100 11010 \\
{98} &00012 11010 20100 &
\sout{99} &00012 01002 10011 &
\sout{100} &00012 10200 11001 \\
\sout{101} &00012 01200 11001 &
{102} &00012 12000 20100 &
{103} &00012 01002 10200 \\
{104} &00012 01020 20100 &
{105} &00012 01020 10200 &
{106} &00012 01101 20100 \\
\sout{107} &00012 01110 10110 &
\sout{108} &00012 01110 10200 &
{109} &00012 02001 10200 \\
{110} &00012 02100 10020 &
\sout{111} &00111 01110 20100 &
{112} &00201 02010 10200 \\
\sout{113} &10011 10200 11010 &
\sout{114} &00102 02010 11100 &
\sout{115} &00102 11100 20010 \\
\sout{116} &01020 01101 11100 &
\sout{117} &00201 01020 11100 &
{118} &00201 02010 20100 \\
{119} &02001 10200 11010 &
{120} &00102 02001 10020 &
 & \\\hline
\end{tabular}

    \end{center}
    \caption{$\dim(\poly)=1$}
  \end{table}
  
  \begin{table}
    \tiny
    \begin{center}
      \begin{tabular}{|l|l||l|l||l|l|}
\hline
 i & $\genset_i$  &  i & $\genset_i$  &  i & $\genset_i$ \\\hline
\hline
\sout{121} &00003 00012 00102 01002 &
\sout{122} &00003 00012 00102 04000 &
\sout{123} &00003 00012 02100 11001 \\
\sout{124} &00003 00012 11001 20100 &
{125} &00003 00012 20100 21000 &
{126} &00003 00012 02100 12000 \\
{127} &00003 00012 02100 21000 &
{128} &00003 00012 12000 20100 &
{129} &00003 00012 01002 00400 \\
{130} &00003 00012 10101 21000 &
{131} &00003 00012 02100 00400 &
{132} &00003 00012 20100 00400 \\
{133} &00003 00012 21000 03100 &
{134} &00003 00012 12000 30100 &
{135} &00003 00012 01200 20100 \\
{136} &00003 00012 01200 04000 &
{137} &00003 00012 20100 01300 &
\sout{138} &00003 00102 02010 11001 \\
\sout{139} &00003 02010 02100 11010 &
\sout{140} &00003 00210 01002 10101 &
\sout{141} &00003 00210 01200 10110 \\
\sout{142} &00003 00210 01110 10200 &
\sout{143} &00003 01002 10101 20010 &
{144} &00003 00102 20010 21000 \\
\sout{145} &00003 00111 01200 20010 &
\sout{146} &00003 00111 02010 10200 &
{147} &00003 01011 02100 20010 \\
\sout{148} &00003 02010 10011 20100 &
{149} &00003 01002 20010 20100 &
{150} &00003 00210 01011 12000 \\
\sout{151} &00003 00210 10011 21000 &
{152} &00003 00210 02100 00031 &
{153} &00003 00210 20100 00031 \\
{154} &00003 02100 21000 00031 &
{155} &00003 12000 20100 00031 &
{156} &00003 12000 20010 00031 \\
{157} &00003 02010 21000 00031 &
{158} &00003 10200 20010 00031 &
{159} &00003 01200 02010 00031 \\
{160} &00003 01200 20100 00031 &
{161} &00003 01200 12000 00031 &
{162} &00003 02100 10200 00031 \\
{163} &00003 10200 21000 00031 &
{164} &00003 00102 02010 21000 &
{165} &00003 01101 02010 20100 \\
{166} &00003 02100 10101 20010 &
{167} &00003 00210 01002 20100 &
{168} &00003 00210 01101 21000 \\
{169} &00003 01002 10200 20010 &
\sout{170} &00003 01200 11001 20010 &
{171} &00003 00102 02010 00040 \\
{172} &00003 02010 20100 00040 &
{173} &00003 02100 20010 00040 &
{174} &00003 00210 01002 00040 \\
{175} &00003 00210 21000 00040 &
{176} &00003 00210 02100 00040 &
{177} &00003 00210 12000 00040 \\
{178} &00003 00210 20100 00040 &
{179} &00003 01002 20010 00040 &
{180} &00003 12000 20100 00040 \\
{181} &00003 02100 21000 00040 &
{182} &00003 01200 20010 00040 &
{183} &00003 01200 02010 00040 \\
{184} &00003 01200 20100 00040 &
{185} &00003 12000 20010 00040 &
{186} &00003 02010 10200 00040 \\
{187} &00003 02010 21000 00040 &
{188} &00003 10200 21000 00040 &
{189} &00003 10200 20010 00040 \\
{190} &00003 02010 20100 00400 &
{191} &00003 02100 20010 00400 &
{192} &00003 00210 21000 04000 \\
{193} &00003 01200 20010 04000 &
{194} &00003 00102 21000 03010 &
{195} &00003 02010 20100 01030 \\
{196} &00003 02100 20010 10030 &
{197} &00003 01002 20100 00310 &
{198} &00003 01101 02010 21000 \\
{199} &00003 10101 12000 20010 &
{200} &00003 00210 21000 00130 &
{201} &00003 00210 01101 20100 \\
{202} &00003 00210 02100 10101 &
{203} &00003 00210 12000 00130 &
\sout{204} &00003 10200 11001 20010 \\
\sout{205} &00003 01200 02010 11001 &
{206} &00003 01002 10200 30010 &
{207} &00003 01200 20010 10030 \\
{208} &00003 02010 10200 01030 &
{209} &00003 02100 20010 01021 &
{210} &00003 02010 20100 10021 \\
{211} &00003 01200 20010 00121 &
{212} &00003 00210 21000 10021 &
{213} &00003 01200 02010 20100 \\
{214} &00003 02100 10200 20010 &
{215} &00003 00210 02100 21000 &
{216} &00003 00210 12000 20100 \\
{217} &00003 01200 12000 20010 &
{218} &00003 02010 10200 21000 &
{219} &00003 02010 20100 01300 \\
{220} &00003 02100 20010 10300 &
{221} &00003 00210 21000 03100 &
{222} &00003 00210 12000 30100 \\
{223} &00003 01200 20010 13000 &
{224} &00003 02010 10200 31000 &
\sout{225} &00003 00111 01020 20100 \\
\sout{226} &00003 00111 01020 20010 &
\sout{227} &00003 00111 01200 20100 &
{228} &00003 01020 02100 20010 \\
{229} &00003 02010 10020 20100 &
{230} &00003 02100 20010 01030 &
{231} &00003 02010 20100 10030 \\
{232} &00003 00120 01020 10110 &
{233} &00003 00120 01011 21000 &
{234} &00003 00120 20010 00301 \\
{235} &00003 00120 01011 20010 &
{236} &00003 00120 21000 00310 &
\sout{237} &00003 00120 02010 10011 \\
{238} &00003 00120 02010 21000 &
{239} &00003 00120 12000 20010 &
{240} &00003 00120 21000 00400 \\
{241} &00003 00120 02010 00400 &
{242} &00003 00120 20010 00400 &
{243} &00003 00120 21000 03010 \\
{244} &00003 00120 12000 30010 &
{245} &00003 00120 01200 20010 &
{246} &00003 00120 01200 10101 \\
{247} &00003 00120 01200 04000 &
{248} &00003 00120 20010 01300 &
{249} &00003 00120 10200 03001 \\
\sout{250} &00003 10011 12000 20100 &
{251} &00003 12000 20010 00130 &
{252} &00003 01200 20010 00130 \\
{253} &00003 00210 01020 20100 &
{254} &00003 00210 01020 04000 &
{255} &00003 00210 20100 01030 \\
{256} &00003 01020 20100 00211 &
{257} &00003 01020 20010 00301 &
{258} &00003 10020 20100 00301 \\
{259} &00003 02010 10020 00301 &
{260} &00003 01020 20100 00310 &
{261} &00003 01020 20100 04000 \\
{262} &00003 01020 02100 00400 &
{263} &00003 10020 20100 00400 &
\sout{264} &02001 02010 02100 11001 \\
\sout{265} &00201 00210 01200 10101 &
\sout{266} &00021 02010 11001 20100 &
\sout{267} &00021 02100 11001 20010 \\
{268} &00021 00210 10101 21000 &
{269} &00021 00210 01101 12000 &
{270} &00021 00210 02100 00004 \\
{271} &00021 00210 20100 00004 &
{272} &00021 02100 21000 00004 &
{273} &00021 12000 20100 00004 \\
{274} &00021 01200 10101 20010 &
{275} &00021 01101 02010 10200 &
{276} &00021 01200 02010 00004 \\
{277} &00021 01200 20100 00004 &
{278} &00021 01200 12000 00004 &
{279} &00021 02010 21000 00004 \\
{280} &00021 10200 21000 00004 &
{281} &00021 12000 20010 00004 &
{282} &00021 02100 10200 00004 \\
{283} &00021 10200 20010 00004 &
\sout{284} &00201 02010 11001 20100 &
\sout{285} &00201 02100 11001 20010 \\
{286} &00210 02001 10101 21000 &
{287} &01200 02001 10101 20010 &
{288} &00210 02001 20100 00004 \\
{289} &00201 02010 21000 00004 &
{290} &00201 12000 20010 00004 &
{291} &02001 10200 20010 00004 \\
\sout{292} &02100 10200 11001 20010 &
\sout{293} &01200 02010 11001 20100 &
{294} &01200 10101 12000 20010 \\
{295} &01101 02010 10200 21000 &
{296} &00210 02100 10101 21000 &
{297} &00210 01101 12000 20100 \\
\sout{298} &00111 01020 10200 20010 &
\sout{299} &00111 01200 10020 20100 &
\sout{300} &00120 00201 10011 21000 \\
{301} &00120 00201 02010 00004 &
{302} &00120 00201 20010 00004 &
{303} &00120 02001 20010 00004 \\
{304} &00120 01011 12000 20010 &
\sout{305} &00120 02010 10011 21000 &
{306} &00120 01200 10101 20010 \\
{307} &00120 01200 02001 00004 &
\sout{308} &01020 02001 10011 20100 &
{309} &02001 10020 20100 00004 \\
{310} &01200 02001 20100 00004 &
{311} &00201 12000 20100 00004 &
{312} &00201 10020 20100 00004 \\
\sout{313} &00210 01020 10011 20100 &
\sout{314} &00021 00111 01020 20010 &
\sout{315} &00021 00111 02010 10020 \\
\sout{316} &00021 00120 01011 20010 &
{317} &00021 00210 01011 10020 &
\sout{318} &00021 00210 01110 20100 \\
{319} &00021 00210 02100 10110 &
{320} &00021 01200 02010 10200 &
\sout{321} &00120 00201 01110 20010 \\
{322} &00120 00201 02010 10110 &
\sout{323} &01020 01110 02001 20010 &
{324} &00021 01200 02010 20100 \\
{325} &00021 02100 10200 20010 &
{326} &00021 00210 02100 21000 &
{327} &00021 00210 12000 20100 \\
{328} &00021 01200 12000 20010 &
{329} &00021 02010 10200 21000 &
{330} &00201 01020 02100 20010 \\
\hline
\end{tabular}

    \end{center}
    \caption{$\dim(\poly)=0$ (part 1/2)}
    
  \end{table}
  
  \begin{table}
    \tiny
    \begin{center}
      \begin{tabular}{|l|l||l|l||l|l|}
\hline
 i & $\genset_i$  &  i & $\genset_i$  &  i & $\genset_i$ \\\hline
\hline
{331} &00201 02010 10020 20100 &
{332} &00120 01200 02001 20010 &
{333} &00210 02001 10020 21000 \\
{334} &00120 00201 02010 21000 &
{335} &00210 01020 02001 20100 &
{336} &00012 00120 21000 00301 \\
{337} &00012 00120 12000 00301 &
{338} &00012 00120 21000 03001 &
{339} &00012 00120 12000 30001 \\
{340} &00012 00120 01200 13000 &
{341} &00012 00120 02001 10300 &
{342} &00012 00120 10200 03100 \\
{343} &00012 02001 20100 01030 &
{344} &00012 02001 20100 01300 &
{345} &00012 00201 21000 00130 \\
{346} &00012 00201 12000 00130 &
{347} &00012 02001 21000 00130 &
{348} &00012 10200 21000 00130 \\
{349} &00012 00201 12000 30100 &
{350} &00012 02100 10200 01120 &
{351} &00012 12000 20100 01300 \\
{352} &00012 00201 01020 13000 &
{353} &00012 00201 20100 01030 &
{354} &00012 00201 02100 10030 \\
{355} &00012 00201 10020 03010 &
{356} &00012 01020 20100 00211 &
{357} &00012 01020 20100 00301 \\
{358} &00012 02100 10020 00301 &
{359} &00012 02001 10020 00310 &
{360} &00012 12000 20100 01030 \\
{361} &00012 01020 12000 01300 &
{362} &00012 01020 10200 03001 &
{363} &00012 01020 10200 30001 \\
{364} &00012 02001 10200 01030 &
{365} &00012 02001 10200 31000 &
{366} &00012 10020 20100 01300 \\
{367} &01020 02001 20100 00013 &
{368} &02001 10020 20100 00013 &
{369} &01020 02001 10200 00013 \\
{370} &01020 10200 20010 00013 &
{371} &01020 02001 10200 30010 &
 & \\\hline
\end{tabular}

    \end{center}
    \caption{$\dim(\poly)=0$ (part 2/2)}
  \end{table}

\end{subtables}

}

\def\polydec{\mathfrak{P}^\ge}
\def\mfQ{\mathfrak{Q}}

The remaining of this section is devoted to the proof of Lemma \ref{lem:371}, which is algorithmic.
In this section, the order of the angles is not important, so we suppose w.l.o.g. that there is an $\alpha=(\alpha_1,\ldots ,\alpha_5)\in \poly_\mX$ such that $1 > \alpha_1 \ge \alpha_2 \ge \alpha_3 \ge \alpha_4 \ge \alpha_5 > 0$.

Let $\polydec_\mX$ be the set of vectors $(\alpha_1,\ldots \alpha_5)$ such that $1 \ge \alpha_1 \ge \alpha_2 \ge \alpha_3 \ge \alpha_4 \ge \alpha_5 \ge 0$, $\sum_i \alpha_i=3$ and for every $v\in \mX$, $v \cdot \alpha = 2$. Clearly, $\polydec_\mX$ is a convex polytope. Moreover, one has $\polydec_\mX = \poly_\mX \cap \polydec_\emptyset$ and $\polydec_\mX = \polydec_{\compat(\mX)}$.
If $\polydec_\mX$ is non empty, let $m_\mX \in [0,1]^5$ be such that $(m_\mX)_i = \min \{ \alpha_i : \alpha \in \polydec_\mX\}$. Note that $(m_\mX)_1\ge \frac{3}{5}$, $(m_\mX)_2\ge \frac{1}{2}$, $(m_\mX)_3\ge \frac{1}{3}$, and $(m_\mX)_i\ge (m_\mX)_{i+1}$ for every $i\in\{1,\ldots 4\}$.

\def\namerecproc{Recurse}

We say that a set $\mX$ is \emph{maximal} if $\mX=\compat(\mX)$. To prove Lemma \ref{lem:371}, it suffices to prove it for every maximal good set.

The procedure \textsc{\namerecproc} (Algorithm \ref{algo:rec}) computes all maximal good sets $\mY\supseteq \mX$ with  $\polydec_{\mY} \cap ]0,1[^5 \ne \emptyset$.

\begin{algorithm}
  \caption{Exhaustive search of goods sets $\mY\supseteq \mX$}
  \label{algo:rec}
  \begin{algorithmic}[1]
\Procedure{\namerecproc}{$\mX$}
\State{$\mX \gets \compat(\mX)$}\label{line:compat}
\IIf{$\polydec_\mX \cap ]0,1[^5 = \emptyset$}\label{line:empty} \Return \EndIIf
\If{$\mX$ is good}
\State{Add $\mX$ to the list of good sets}
\EndIf
\State{Let $u\in\mathbb{R}^5$ such that: \label{line:choiceu} \begin{itemize} \item $u\cdot \v1=0$ \item $\forall v\in \mX$, $u\cdot v = 0$ and \item $\forall i\in \{4, 5\}$, $({m_\mX})_i=0 \Rightarrow u_i<0$\end{itemize}}
\State{$V\gets \{v\in \mathbb{N}^5 : v\cdot u\ge 0 \text{ and } v\cdot {m_\mX} \le  2\}$ \label{line:vm0}}
\For{every $w\in V\setminus \mX$} \State \Call{\namerecproc}{$\mX \cup \{w\}$} \label{line:rec} \EndFor
\EndProcedure
\end{algorithmic}
\end{algorithm}

For all maximal good set $\mY\supseteq \mX$, one has $\compat(\mX)\subseteq \mY$ (line \ref{line:compat}).
For $u$ defined line \ref{line:choiceu}, since $u\cdot \v1 =0$, for every $v\in\mX$, $v\cdot u=0$, and by definition of good subsets, we know that
if there is another maximal good set $\mY \supsetneq \mX$, then there is a $w\in \mathbb{N}^5\setminus \mX$ such that $w\cdot u\ge 0$. Moreover, we must have $w\cdot m_\mX \le 2$, otherwise $\poly_\mY\cap ]0,1[^5$ would be empty. Thus, $w$ is in the set $V$ computed line \ref{line:vm0}. We try every possibility for $w$ at line \ref{line:rec}.
An important point of the algorithm is that $V$ is finite: if $v\in V$, then for every $i$ such that $(m_\mX)_i>0$, $v_i$ is bounded by $\frac{2}{m_\mX}$, and thus for every $i$ with $(m_\mX)_i=0$, $v_i$ is bounded by $-\frac{1}{u_i} \sum_{j:(m_\mX)_j>0} (v_j\max(0,u_j))$.

The computation of a $u$ (line \ref{line:choiceu}) with the required property is done using a linear program. 
If no such $u$ exists, then the algorithm would fail. But even it is not necessary for the proof (it suffices that one possible execution terminates), one can show that this case never happens.

\begin{proposition}
  In line \ref{line:choiceu}, such a $u$  always exists.
\end{proposition}
\begin{proof}
Note that if $\alpha,\alpha'\in \poly_{\mX}$, then $(\alpha-\alpha')\cdot \v1 =0$ and for every $v\in \mX$, $(\alpha-\alpha')\cdot v =0$.
If $({m_\mX})_4>0$ and $({m_\mX})_5>0$, one can take $u=(0,0,0,0,0)$.
If $({m_\mX})_4>0$ and $({m_\mX})_5=0$, there is $\alpha\in \polydec_\mX$ such that $\alpha_5=0$. %
Otherwise $({m_\mX})_4=({m_\mX})_5=0$, and there is $\alpha\in \polydec_\mX$ such that $\alpha_4=\alpha_5=0$.
In all cases, $u=\alpha-\alpha'$, with $\alpha' \in \polydec_\mX \cap ]0,1[^5 \ne \emptyset$, has the desired properties.
\end{proof}

The procedure \textsc{\namerecproc} is non deterministic, and a good choice for $u$ can reduce the size of the research tree and the computation time.
But since the dimension of the subspace spanned by $\mX$ strictly increase at every recursive call, there is at most $5$ nested calls of \textsc{\namerecproc}, and this procedure always terminates.

\paragraph{Computation.}
In order to reduce the computation time, we also track vectors $w$ which are not in $\mY$.
Our implementation takes approximately 40 seconds to explore all the cases (1354 calls of {\sc Recurse}).
There are 193 non-empty maximal goods sets $\mX$ with $\poly^\ge_\mX\cap ]0,1[^5\ne\emptyset$, and (taking all permutations) 3495 sets with $\poly_\mX\cap ]0,1[^5\ne\emptyset$.
If we keep only one representative  for each class up to rotation/mirror, one has the 371 sets of Tables~\ref{tab:371}.

\section{Testing all 371 cases}\label{sec:phase3}

\def\azero{\emptyset}
\def\api{\pi}
\def\aunk{\operatorname{unknown}}

For each of the 371 cases, we do an exhaustive search by backtracking, to try to construct a tiling of an arbitrarily large region.
If this backtracking is finite, then we know that there is no pentagon with these angles condition which tiles the plane.

\medskip

Throughout this section one fix $\mX=\compat(\genset_i)$ for an $i\in\{1,\ldots 371\}$.
One supposes that $\penta$ is a pentagon which tiles the plane with a tiling $\tiling$ of positive density, such that $\compat(\setvecctype(\tiling))=\mX$. %
Let $\poly = \poly_\mX$.

Suppose that the vertices of $\penta$ are $s_i$ (with angle $\alpha_i\times \pi$) and the lengths of the sides are $\ell_i$, $i\in\{1,\ldots 5\}$, in clockwise order, and such that $\ell_1$ is the length between $s_1$ and $s_2$. Moreover, we suppose w.l.o.g. that $\sum_i \ell_i=1$.

\medskip

The backtracking is done on a pair of two data-structures:
a tiling graph which represents the geometric information we have for the part of the tiling,
and a linear program $Q$ with represent conditions we have on $\ell$.

\paragraph{Tiling graph.}
The \emph{tiling graph} is an embedded planar graph, with additional information: labels on angles, edge and faces.
(Note that this graph differs significantly from the graph defined in Section \ref{sec:positive}.)

Each vertex of the graph corresponds to a vertex of the tiling. (This mapping is not necessarily injective.)
Each angle has a type: either $1$, $2$, $3$, $4$, $5$, $\azero$, $\api$ or $\aunk$.
Each edge in the graph has also a type: either $1$, $2$, $3$, $4$ or $5$.
The planar graph has two types of faces. A face is either a \emph{normal} face or is a \emph{special} face.
Each edge is adjacent to one special face, and one normal face.

There is a bijection between the normal faces and the tiles of the tiling. %
Thus a normal face has degree 5, and the types of its angles and edges are either (in clockwise order) $1$,$2$,$3$,$4$,$5$ or $5$,$4$,$3$,$2$,$1$. Moreover the type of the edge between the angles $1$ and $2$ is $1$.

A special face corresponds either to the frontier between tiles or an unknown area of the plane.
Its angles are $\azero$, $\api$ or $\aunk$. An angle of type $\azero$ (resp. $\pi$) corresponds to an angle of $0$ (resp. $\pi$) in the tiling.
A special face is \emph{complete} if it has no $\aunk$ angles.
A complete special face has exactly two $\azero$ angles. In this case, it corresponds to a segment which is a frontier between two or more tiles.

A vertex $v$ is \emph{complete} if there is no $\aunk$~angle adjacent to it.
Similarly to Section \ref{sec:positive}, let $\cvect(v)\in \mathbb{N}^5$ be such that, for every $i\in\{1,\ldots 5\}$, $(\cvect(v))_i=c\times (\vect(v))_i$, where $(\vect(v))_i$ is the number of angles $i$ adjacent to $v$, and $c=1$ if there is no $\api$ adjacent to $v$ (\emph{i.e.} $v$ is full) and $c=2$ if there is one angle $\api$ adjacent to $v$ (\emph{i.e.} $v$ is half).

An example of a tiling graph is given in Figure \ref{fig:tgraph}. The tiling graph on the right corresponds to the tiling on the left. Note that other tiling graphs are possible to represent the same tiling.

\tikzset{
    position/.style args={#1:#2 from #3}{
        at=(#3.#1+45), anchor=#1+180+45, shift=(#1+45:#2)
    }
}
\newdimen\nodeDist
\nodeDist=.5mm

\begin{figure}
\begin{minipage}{.3\linewidth}
\begin{tikzpicture}[scale=1., rotate=45]
\draw (-0.585253,-0.214217) -- (0.000000,0.799471) ;
\draw (-0.428436,-0.799471) -- (-0.585253,-0.214217) ;
\draw (0.428435,0.057399) -- (-0.428436,-0.799471) ;
\draw (0.585253,0.642652) -- (0.428435,0.057399) ;
\draw (0.000000,0.799471) -- (0.585253,0.642652) ;

\draw (-0.428436,1.227905) -- (0.000000,0.799471) ;
\draw (-0.856871,0.799471) -- (-0.428436,1.227905) ;
\draw (-1.170506,-0.371036) -- (-0.856871,0.799471) ;
\draw (-0.585253,-0.214217) -- (-1.170506,-0.371036) ;
\draw (0.000000,0.799471) -- (-0.585253,-0.214217) ;

\draw (-0.428436,1.227905) -- (-0.856871,0.799471) ;
\draw (-0.856871,1.656340) -- (-0.428436,1.227905) ;
\draw (-2.027377,1.969977) -- (-0.856871,1.656340) ;
\draw (-1.870559,1.384724) -- (-2.027377,1.969977) ;
\draw (-0.856871,0.799471) -- (-1.870559,1.384724) ;

\draw (-1.598942,2.398412) -- (-0.428436,2.398412) ;
\draw (-2.027377,1.969977) -- (-1.598942,2.398412) ;
\draw (-0.856871,1.656340) -- (-2.027377,1.969977) ;
\draw (-0.271618,1.813159) -- (-0.856871,1.656340) ;
\draw (-0.428436,2.398412) -- (-0.271618,1.813159) ;

\draw (-0.271618,1.813159) -- (0.742070,1.227905) ;
\draw (-0.856871,1.656340) -- (-0.271618,1.813159) ;
\draw (0.000000,0.799471) -- (-0.856871,1.656340) ;
\draw (0.585253,0.642652) -- (0.000000,0.799471) ;
\draw (0.742070,1.227905) -- (0.585253,0.642652) ;
\end{tikzpicture}
\end{minipage}
\begin{minipage}{.7\linewidth}
\begin{tikzpicture}[scale=2.5,every node/.style={circle,inner sep=0pt,text width=4mm,align=center}, rotate=45]
\input{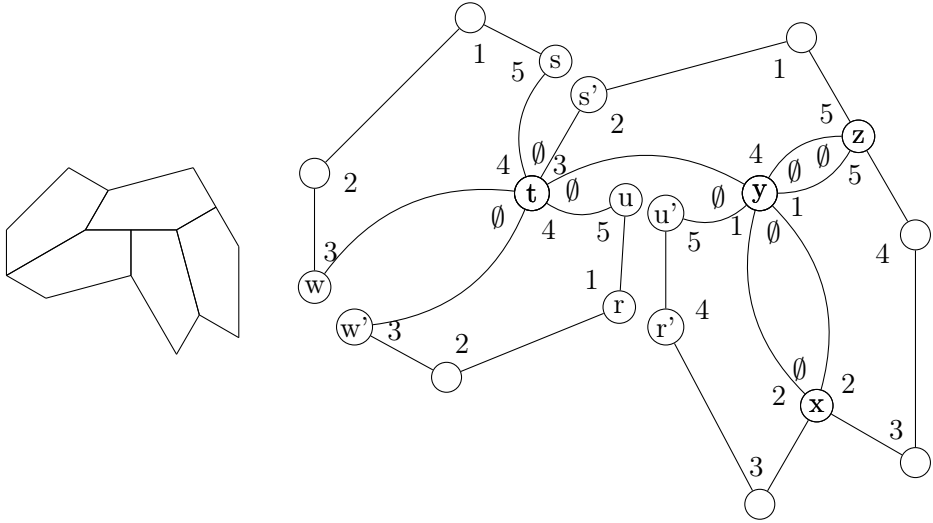}
\end{tikzpicture}
\end{minipage}
\caption{Example of a tiling graph (Type 15). Unmarked angles are labeled ``$\aunk$''.}
\label{fig:tgraph}
\end{figure}

A \emph{run} on a special face is a succession of consecutive $\azero$ and $\api$ angles. Each run corresponds to aligned points in the tiling.
For example, on Figure \ref{fig:tgraph}, $(s,t,s')$, $(u,t,y,u')$ are (maximal) runs. The complete face $(y,z)$ induces also a run.

Since we want to generate a tiling graph corresponding to a tiling by $\penta$, we keep the following conditions on the tiling graph (that is, we backtrack if one of these conditions are not fulfilled)
\begin{itemize}
  \item for every vertex $v$, there is a $w\in \mX$ such that $\cvect(v)\le w$,
  \item for every complete vertex $v$, $\cvect(v) \in \mX$,
  \item there is no run with more than two $\azero$ angles,
  \item there is no vertex with two $\api$ angles adjacent to it.
\end{itemize}
Note that every finite subset $\tiling'$ of $\tiling$ can ge represented by a tiling graph with the previous properties (but the representation is not unique).
We will make some ``completion'' operations on it, which guarantee that the new tiling graph is also a tiling graph of the same tile set.

Moreover, during the exploration, all the operations we do on the tiling graph keep the additional following conditions: the graph is connected, has exactly one non-complete face, and has no vertices with most than one $\aunk$~angles adjacent to it.

\paragraph{Completing vertices.}
At every time, as soon as there is a non-complete vertex $v$ such that $\cvect(v)\in \setvecctype$, we relabel the angle labeled $\aunk$ adjacent to $v$ with the label $\azero$.
Moreover, for every non-complete vertex $v$ such that $2\times \vect(v)\in \setvecctype$, we relabel the angle labeled $\aunk$ adjacent to $v$ with the label $\api$.

\paragraph{Length suppositions and completing faces.}
If there is a pair of vertices $(v,v')$ on a same run, and the linear program $Q$ imply that $v$ and $v'$ are the same point in the tiling, then we merge $v$ and $v'$ in the graph.

If $Q$ does not permit to decide among the following 3 possibilities :
\begin{itemize}
\item $v$ and $v'$ are the same point in the tiling,
\item $v$ is on the right of $v'$ (with an arbitrary orientation of the line corresponding to the run),
\item $v$ is on the left of $v$,
\end{itemize}
then we branch on the 3 possibilities: we add the corresponding condition on $Q$, and recurse.

\begin{example}
In Figure \ref{fig:tgraph}, $(w,t,w')$ is a run, and the length between $t$ and $w$ (which is $\ell_3$) is the same as the distance between $t$ and $w'$. Thus we merge $w$ and $w'$, and we relabel the angle $t,w,t$ into $\azero$, and we create a new special complete face $(t,w)$.

We have also to consider the run $(u,t,y,u')$. We have either to choose if $u$ and $u'$ is the same point (that is, add the condition $\ell_3=\ell_4+\ell_5$) or not (and in this case, explore with the additional condition $\ell_3>\ell_4+\ell_5$).

Suppose now we explore the first case (merge of $u$ and $u'$). We create an angle $\api$ adjacent to $u$ to complete the special face $(u,t,y)$. Since $2\alpha_5=\pi$ in the type 15 (i=303), the vertex $u$ is now complete, and we can label the angle $r,u,r'$ as $\azero$. We have a new run $(r,u,r')$. Since we have already the condition $\ell_4=\ell_5$ in $Q$ (by the complete special face $y,z$), we know that $r$ and $r'$ must be the same vertex in the tiling, and we can merge $r$ and $r'$ without branching.
\end{example}

\paragraph{Existence of a solution which respects Q}

Let $s(\alpha)$ be the vector such that $s(\alpha)_i=(i-1) -\sum_{j=1}^{i-1} \alpha_i$.
Note that if $\alpha$ (resp. $\ell$) is the vector of angles (resp. lengths) of a pentagon, we have
\begin{equation}\label{eq:eq}
\sum_i \ell_i \exp(s(\alpha)_i \times \pi \times \sqrt{-1} )=0.
\end{equation}

Given a linear program $Q$, we denote by $\mfQ$ the set of $\ell\in\mathbb{R}^5$ such that $\ell\ge 0$, $\sum \ell=1$ and $\ell$ respects the conditions in $Q$.

If the following condition is not fulfilled, then no convex pentagon exists with the conditions, and one backtrack.
\begin{equation}\label{eq:nopoly}
  \exists \ell\in \mfQ \cap ]0,1[^5, \exists \alpha\in\poly \cap ]0,1[^5, \sum_i \ell_i \exp(s(\alpha)_i\times \pi \times \sqrt{-1})= 0
\end{equation}

If $\dim(\poly)=0$ (cases from 121 to 371), since all conditions for $\poly$ are rational, there is a $p\in\mathbb{Z}^5$ and $q\in \mathbb{N}^+$ such that $s(\alpha)={p}/{q}$, where $\{\alpha\}=\poly$, and thus the condition $\sum_i \ell_i \exp(s(\alpha)_i \times \pi \times \sqrt{-1})= 0$ can be turned into $\ell \cdot \cos =0$ and $\ell \cdot \sin =0$, where $\cos_i=\cos(p_i \times {\pi}/{q})$ and $\sin_i=\sin(p_i \times {\pi}/{q})$.
The one can decide with computations on an algebraic extension of $\mathbb{Q}$ (for example $\mathbb{Q}[\cos\left({\pi}/{q}\right)]$).

\medskip

If $\dim(\poly)>0$, the verification of (\ref{eq:nopoly}) is more complicated, and we do not try to verify it at each recursion.
However, if one has a certificate that (\ref{eq:nopoly}) is false, then we backtrack.
If there is a family of polytopes
$\poly_l$, $l\in \{1,\ldots,L\}$
such that $\poly\subseteq \bigcup_{l=1}^L \poly_l$, and for every $l\in \{1,\ldots,L\}$, $\{x\in \mfQ : x\cdot \sin^+ \ge 0, x\cdot \sin^- \le 0,x\cdot \cos^+ \ge 0 \text{ and } x\cdot \cos^- \le 0\} = \emptyset$, where $\sin^+_i$ (resp. $\sin^-_i$) is an upper (resp. lower) bound of $\{\sin(\pi s(\alpha)_i) : \alpha \in \poly_l\}$ (and similarly for $\cos^+$ and $\cos^-$), then we know that (\ref{eq:nopoly}) if false. This can be done using rational numbers.

This procedure cannot decide, for example, if (\ref{eq:nopoly}) is false, but (\ref{eq:eq}) has a degenerate solution $(\ell,\alpha)$ is on the boundary of $\mfQ \times \poly$.
To resolve these cases, we also backtrack in some degenerate solutions (types 20 to 24 in Table \ref{table:types}).

\paragraph{Branching.}
If we are not in any case of backtracking, then we add a new normal face to the tiling graph. 

We take a non-complete vertex $w$ in the graph. We know that, if the tiling graph corresponds to a sub-tiling $\tiling'$ of a tiling $\tiling$ by $\penta$, there is a tile $P\in \tiling \setminus \tiling'$ such that $w$ is a vertex of $P$, and $P$ shares a line segment with $\tiling'$.
Then we branch on on all theses possibilities of face addition.

\begin{table}
\scriptsize
  \begin{tabularx}{\textwidth}{|p{1.1cm}|X|p{2.2cm}||p{1.1cm}|X|p{2.2cm}|}
\hline
Type 1\newline(i=1)&
$a+b+c=2\pi$&
&
Type 2\newline(i=2)&
$a+b+d=2\pi$&
$C=E$\\\hline
Type 3\newline(i=31)&
$3 e=2\pi$\newline 
$d+2 e=2\pi$\newline 
$b+2 e=2\pi$&
$C+E=D$\newline 
$A=B$&
Type 4\newline(i=6)&
$a+b+d=2\pi$\newline 
$2 e=\pi$&
$D=E$\newline 
$B=C$\\\hline
Type 5\newline(i=4)&
$3 e=2\pi$\newline 
$a+b+d=2\pi$&
$D=E$\newline 
$B=C$&
Type 6\newline(i=13)&
$d+2 e=2\pi$\newline 
$a+c+d=2\pi$&
$C=D=E$\newline 
$A=B$\\\hline
Type 7\newline(i=17)&
$d+2 e=2\pi$\newline 
$a+2 c=2\pi$&
$A=C=D=E$&
Type 8\newline(i=14)&
$d+2 e=2\pi$\newline 
$2 b+c=2\pi$&
$A=B=C=D$\\\hline
Type 9\newline(i=15)&
$d+2 e=2\pi$\newline 
$2 a+c=2\pi$&
$A=B=C=D$&
Type 10\newline(i=69)&
$2 c+d=2\pi$\newline 
$b+c+e=2\pi$\newline 
$a+2 b=2\pi$&
$A+C=D=E$\\\hline
Type 11\newline(i=67)&
$c+2 d=2\pi$\newline 
$b+d+e=2\pi$\newline 
$a+2 b=2\pi$&
$A=B=C+2 E$&
Type 12\newline(i=67)&
$c+2 d=2\pi$\newline 
$b+d+e=2\pi$\newline 
$a+2 b=2\pi$&
$A+C=B=2 E$\\\hline
Type 13\newline(i=63)&
$b+2 d=2\pi$\newline 
$a+b+d=2\pi$\newline 
$2 e=\pi$&
$A=2 B=2 C$&
Type 14\newline(i=67)&
$c+2 d=2\pi$\newline 
$b+d+e=2\pi$\newline 
$a+2 b=2\pi$&
$A=B=2 C=2 E$\\\hline
Type 15\newline(i=303)&
$c+2 d=2\pi$\newline 
$2 b+e=2\pi$\newline 
$2 a+d=2\pi$\newline 
$2 e=\pi$&
$B=D=E$\newline 
$C=2 B$&
Type 16\newline(i=72)\newline $\subset$ T10&
$b+c+e=2\pi$\newline 
$2 b+d=2\pi$\newline 
$a+2 c=2\pi$&
$2 A=D=E$\newline 
$A=C$\\\hline
Type 17\newline(i=25)\newline $\subset$ T2&
$c+2 e=2\pi$\newline 
$2 b+d=2\pi$&
$A=B=C=D=E$&
Type 18\newline(i=73)\newline $\subset$ T2&
$d+2 e=2\pi$\newline 
$c+2 e=2\pi$\newline 
$b+d+e=2\pi$&
$D=E$\newline 
$A=B$\\\hline
Type 19\newline(i=23)\newline $\subset$ T1&
$c+2 e=2\pi$\newline 
$b+2 d=2\pi$&
$A=B=C=D$&
Type 20\newline(i=2)\newline degen.&
$a+b+d=2\pi$&
$A=C+D$\newline 
$B=E$\\\hline
Type 21\newline(i=12)\newline degen.&
$d+2 e=2\pi$\newline 
$2 a+b=2\pi$&
$A=B$\newline 
$C=D$&
Type 22\newline(i=27)\newline degen.&
$c+2 e=2\pi$\newline 
$a+2 d=2\pi$&
$A=B=C=E$\\\hline
Type 23\newline(i=64)\newline degen.&
$2 b+d=2\pi$\newline 
$a+b+d=2\pi$\newline 
$2 e=\pi$&
$A=2 C=2 D$&
Type 24\newline(i=69)\newline degen.&
$2 c+d=2\pi$\newline 
$b+c+e=2\pi$\newline 
$a+2 b=2\pi$&
$2 D=A+C$\newline 
$2 E=A+C$\\\hline

\end{tabularx}
  \caption{Conditions for tilings of types 1 to 24, with $(a,\ldots, e)=(\alpha_1,\ldots, \alpha_5)$ and $(A,\ldots, E)=(\ell_1,\ldots, \ell_5)$.}\label{table:types}
\end{table}

\paragraph{Results.}
If we also backtrack if we are in one the 24 types presented in Table \ref{table:types}, the exhaustive search terminates, for all 371 cases for angle conditions.
That is, if a convex pentagon $\penta$ tiles the plane, then $\penta$ is in one the 24 families.

Types 1 to 15 are the already known families of pentagons which tiles the plane. 
Types 16 to 19 are special cases of the 15 already known families, \emph{i.e.} every non-degenerate solution of (\ref{eq:nopoly}) is in a known family. 
Type 16 is a special case of Type 10: the only solution of (\ref{eq:nopoly}) has $\alpha_1=\alpha_4=\alpha_5=\frac{\pi}{2}$ and $\alpha_2=\alpha_3=\frac{3\pi}{4}$.
Types 17 and 18 are special cases of Type 2: in every solution, all lengths are equal and $\alpha_1+\alpha_3=\pi$.
Type 19 is a special case of Type 1: conditions imply that for every solution, $\alpha_3+\alpha_4=\pi$.
Finally, types 20 to 24 are degenerate, \emph{i.e.} (\ref{eq:nopoly}) has no solution. These observations can be done using a computer algebra system, turning linear conditions on angles, lengths and (\ref{eq:eq}) into a system of polynomial equations.

\begin{figure}
\begin{tikzpicture}[scale=.7]
\input{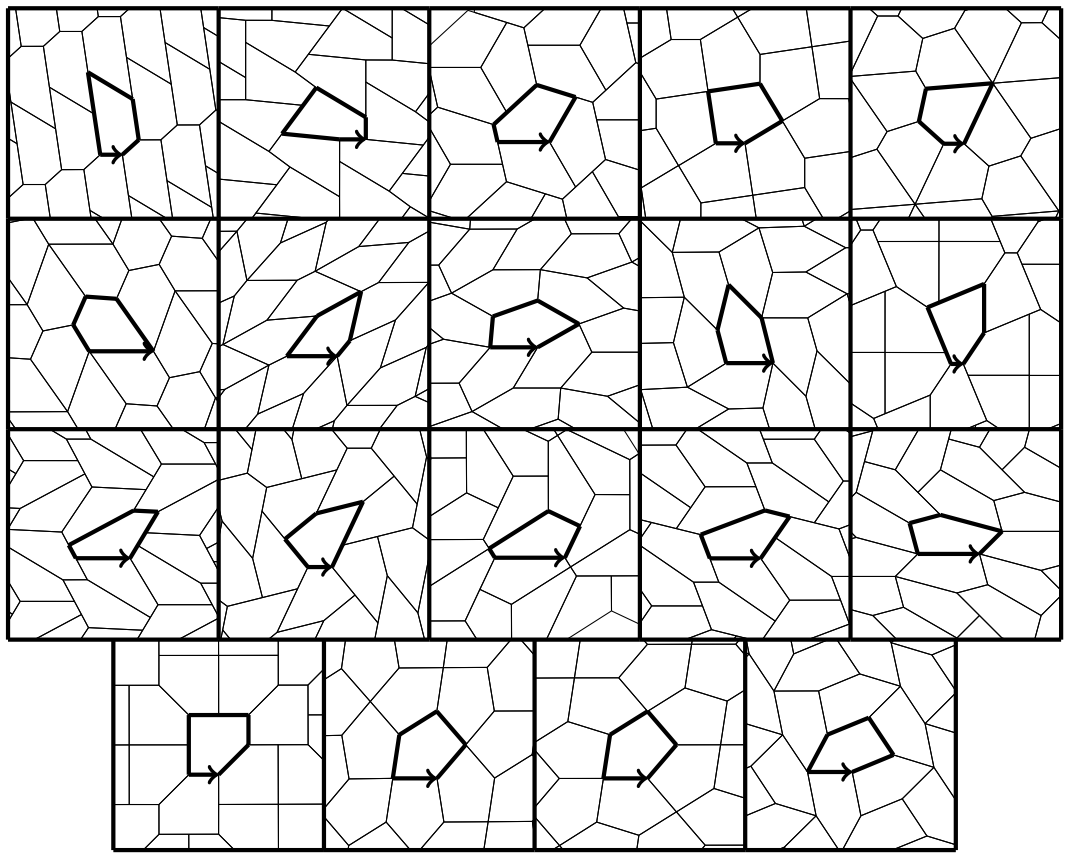}
\end{tikzpicture}
\caption{Tilings 1 to 19. The arrow is from vertex $s_1$ to $s_2$ in the bold tile.}
\end{figure}

\end{document}